\documentclass[graybox]{svmult}

\usepackage{mathptmx}       
\usepackage{helvet}         
\usepackage{courier}        
\usepackage{type1cm}        
\usepackage{makeidx}         
\usepackage{graphicx}        
\usepackage{multicol}        
\usepackage[bottom]{footmisc}
\usepackage{amssymb}
\usepackage{amsmath}

\makeindex             

\newcommand{\cT}{\mathcal{T}}
\newcommand{\cR}{\mathcal{R}}
\newcommand{\cL}{\mathcal{L}}

\newcommand{\FR}{{{}^\bullet}\mathbb{R}}
\newcommand{\triangleL}{\triangleleft_{\mathcal{L}}}
\newcommand{\triangleLq}{\trianglelefteq_{\mathcal{L}}}
\newcommand{\triangleR}{\triangleright_{\mathcal{R}}}
\newcommand{\triangleRq}{\trianglelefteq_{\mathcal{R}}}

\begin{document}

\title*{Nests, and their Role in the Orderability Problem}

\author{Kyriakos Papadopoulos}

\institute{Kyriakos Papadopoulos \at College of Engineering, American University of the Middle East, Egaila, Kuwait \email{kyriakos.papadopoulos1981@gmail.com}}

\maketitle


\abstract{This chapter is divided into two parts. The first part is a survey of some recent results on
nests and the orderability problem. The second part consists of results, partial results and open questions,
all viewed in the light of nests. From connected LOTS, to products of LOTS and function spaces, up to an
order relation in the Fermat Real Line.}

\section{Introduction.}
\label{sec:1}

{\it ``...confusion connotes something which possesses no {\bf order}, \\the individual
parts of which are so strangely admixed and interwined, \\that it is impossible to detect
where each element actually belongs...''} \\(Extract from The Musical Dialogue, by {\it Nikolaus Harnoncourt}, Amadeus Press, 1997.)\\

What is an orderability theorem? In particular in S. Purisch's account of results on orderability and suborderability (see \cite{Purisch-History}),
one can read the formulation and development of several orderability theorems, starting from
the beginning of the 20th century and reaching our days. By an orderability theorem, in topology, we mean the following.
Let $(X,\cT)$ be a topological space. Under what
conditions does there exist an order relation $<$ on $X$ such that the topology $\cT_<$ induced
by the order $<$ is equal to $\cT$? As we can
see, this problem is very fundamental as it is of the same weight as the metrizability problem, for example
(let $X$ be a topological space: is there a metric $d$, on $X$, such that the metric topology generated
by this metric to be equal to the original topology of $X$?).

\section{Some History.}
\label{sec:2}

{\it ``{\bf Order} is a concept as old as the idea of number \\ and much of early mathematics was devoted
to \\ constructing and studying various subsets of the \\ real line.''} (Steve Purisch \cite{Purisch-History})

The great German mathematician Georg Cantor (1845-1918) is credited to be one of the inventors
of set theory. This fact makes him automatically one of the inventors of order-theory as well, as he is the one
who first introduced the class of cardinals and the class of ordinal numbers, two classes of rich
order-theoretic properties. Cantor was not only interested in defining classes of ordered sets, and studying their arithmetic;
he also produced major results while examining order-isomorphisms, that is, bijective
order-preserving mappings between sets whose inverses are also order-preserving. S. Purisch gives a complete list of
these historic papers written by Cantor, in his article ``A History of Results on Orderability
and Suborderability'' \cite{Purisch-History}.

Together with set theory, the field of topology met a rapid rising in the early $20^{th}$ century
and new problems, combining both fields, appeared. A topologist's temptation is always to
examine what sort of topology can be introduced in a given set. So, a very early question was what is the relationship between the natural topology of a set and the topology which is induced by an ordering in this set; this question led
to the formulation of the {\em orderability problem}.

According to Purisch, one of the earliest orderability theorems was introduced by  O. Veblen and N.J. Lennes,
who were both students of the American mathematician E.H. Moore (1862-1932), and who attended his geometry seminar.
This theorem stated that {\em every metric continuum, with exactly two non-cut points, is homeomorphic
to the unit interval}. For the statement of the theorem, Veblen combined the notions of
ordered set and topology, for defining a simple arc. Lennes used up-to-date machinery to
prove Veblen's statement, a proof that was published in 1911.

In the meanwhile, some of the greatest mathematicians of the first half of the 20th century,
like the French mathematicians R. Baire, M. Fr\'{e}chet, the Dutch mathematician L.E.J. Brouwer, the Jewish-German mathematician F. Hausdorff,
the Polish mathematicians S. Mazurkiewicz, W. Sierp\'{i}nski, the Russian mathematicians P. Alexandroff and P. Urysohn and others, were devoted to constructing various subsets
of the real line. In particular, Baire used ideas of the Yugoslavian mathematician D. Kurepa and of the Dutch mathematician A.F. Monna, on non-Archimedean spaces, in order to
characterise the set of irrational numbers. The British mathematician, A.J. Ward, found a topological characterisation of the real line
 (1936), stating that {\em the real line is homeomorphic to a separable, connected and locally connected metric space $X$,
 such that $X-\{p\}$ consists of exactly two components, for every $p \in X$}.

A more general result (1920), by Mazurkiewicz and Sierp\'{i}nski, stated
that {\em compact, countable metric spaces are homeomorphic to well-ordered sets}; this is one of the first, if not
the first, topological characterisation of abstract ordered sets.

Having in mind that a special version of the orderability problem was solved in the beginning of the 70s (J. van Dalen and E. Wattel), its formulation
started from the beginning of the 40s. In particular, the Polish-American mathematician S. Eilenberg,
gave in 1941 the following result: {\em a connected space, $X$, is weakly orderable, if and only if
$X \times X$ minus the diagonal is not connected}.
This condition is also necessary and sufficient for a connected, locally connected space to be orderable.

The American mathematician E. Michael extended
this work and showed, in 1951, that {\em a connected Hausdorff space $X$ is a weakly orderable space, if and only
if $X$ admits a continuous selection}.

It took two more decades, for a complete topological characterization of GO-spaces and LOTS to appear. In
1972 J. de Groot and P.S. Schnare showed \cite{de-groot} that {\em a compact $T_1$ space $X$ is LOTS, if and only if
there exists an open subbase $\mathcal{S}$ of $X$ which is the union of two nests, such that every cover
of the space, by elements of $\mathcal{S}$, has a two element subcover}. {\bf J. van Dalen and E. Wattel}
used the characterisation of de Groot and Schnare as a basis for their construction, which led to
a solution of the orderability problem via nests. We revisited van Dalen and Wattel's characterization
in \cite{Good-Papadopoulos}, and we introduced a simpler proof of their main characterization theorem.

The study of ordered spaces did not finish with the solution to the orderability problem that
was proposed by van Dalen and Wattel. On the contrary, many interesting and important results have
appeared since then. We will now refer to those results which have motivated our own research
in particular.

In 1986, G.M. Reed published an article with title ``The Intersection Topology w.r.t. the Real Line
and the Countable Ordinals'' \cite{Reed}. The author constructed there a class which was shown
to be a surprisingly useful tool in the study of abstract spaces. We know that,
if $\cT_1,\cT_2$ are topologies on a set $X$, then the {\em intersection topology}, with respect
to $\cT_1$ and $\cT_2$, is the topology $\cT$ on $X$ such that the set $\{U_1 \cap U_2: U_1 \in \cT_1 \textrm{ and } U_2 \in \cT_2\}$
forms a base for $(X,\cT)$. Reed introduced the class $\mathcal{C}$, where $(X,\cT) \in \mathcal{C}$ if and only if $X = \{x_\alpha: \alpha < \omega_1\} \subset \mathbb{R}$, where $\cT_1 = \cT_{\mathbb{R}}$ and $\cT_2 = \cT_{\omega_1}$ and $\cT$ is the
intersection of $\cT_{\mathbb{R}}$ (the subspace real line topology on $X$) and $\cT_{\omega_1}$ (the order
topology on $X$, of type $\omega_1$). In particular, Reed showed that if $(X,\cT) \in \mathcal{C}$,
then $X$ has rich topological, but not very rich order-theoretic properties. In particular, $X$ is a completely regular, submetrizable, pseudo-normal, collectionwise Hausdorff, countably metacompact,
first countable, locally countable space, with a base of countable order, that is neither subparacompact,
metalindel{\"o}f, cometrizable nor locally compact. That an $(X,\cT) \in \mathcal{C}$ does not necessarily have
rich order-theoretic properties comes from the fact that there exists, in ZFC, an $(X,\cT) \in \mathcal{C}$
which is not normal.

Eric K. van Douwen characterised in 1993 \cite{van-douwn-misnamed} the noncompact spaces, whose
every noncompact image is orderable, as the noncompact continuous images of $\omega_1$. Van Douwen
refers to a closed non-compact set as {\em cub} (corresponding to closed unbounded sets in ordinals - also met as club in the literature),
and he calls {\em bear} a space which is noncompact and has no disjoint cubs.
Here we state his result that has motivated our research on ordinals (see \cite{Good-Papadopoulos}):\\
For a noncompact space $X$, the following are equivalent:
\begin{enumerate}
\item $X$ is a continuous image of $\omega_1$.

\item Every noncompact continuous image of $X$ is orderable.

\item $X$ is scattered first countable orderable bear.

\item $X$ is locally countable orderable bear.

\item $X$ has a compatible linear order, all initial closed segments
of which are compact and countable.
\end{enumerate}

\section{A Survey of Recent Results on Nests.}
\label{sec:3}

\subsection{Characterizations of LOTS.}

As we also mentioned in Section \ref{sec:2}, van Dalen and Watten used nests in order
to give a solution to the orderability problem, and in \cite{Good-Papadopoulos}
we gave a more order- and set-theoretic dimension to this characterization. In particular,
we did not declare our space being $T_1$, but its topology generated by a (so-called)
$T_1$-separating subbase.

\begin{definition}\label{separating-families}
Let $X$ be a set. We say that a collection of subsets $\mathcal{S}$ of $X$:
\begin{enumerate}
\item
$T_0$-{\em separates $X$}, if and only if for all $x,y \in X$, such that $x \neq y$, there exists $S \in
\mathcal{S}$, such that $x \in S$ and $y \notin S$ or $y \in S$
and $x \notin S$,

\item
$T_1$-{\em separates $X$}, if and only if for all $x,y \in X$, such that $x \neq y$, there exist $S,T \in
\mathcal{S}$, such that $x \in S$ and $y \notin S$ and also $y \in
T$ and $x \notin T$ and

\end{enumerate}
\end{definition}

One can easily see that a space is $T_0$ (resp. $T_1$) if and only if its topology is generated by
a $T_0$- (resp. $T_1$-) separating subbase, but the statement of Definition \ref{separating-families} is not
valid for the $T_2$ separation axiom, if one defines a $T_2$-separating subbase in an analogous way.

\begin{definition}
Let $X$ be a set and let $\cL \subset \mathcal{P}(X)$. We define an order
$\triangleL$ on $X$ by declaring that $x \triangleL y$,
if and only if there exists some $L \in \cL$, such that $x \in L$ and $y \notin L$.
\end{definition}

In \cite{Good-Papadopoulos} we showed the close link between nests and linear orders in Theorem \ref{theorem - for Theorem Dalen-Wattel} that follows below.

\begin{theorem}\label{theorem - for Theorem Dalen-Wattel}
Let $X$ be a set and let $\cL \subset \mathcal{P}(X)$. Then, the following hold:
\begin{enumerate}

\item If $\cL$ is a nest, then $\triangleL$ is a transitive relation.

\item $\cL$ is a nest, if and only if for every $x,y \in X$, either
$x=y$ or $x \ntriangleleft_\mathcal{L} y$ or $y \ntriangleleft_\mathcal{L} x$.

\item $\cL$ is $T_0$-separating, if and only if for every $x,y \in X$, either
$x=y$ or $x \triangleL y$ or $y \triangleL x$.

\item $\cL$ is a $T_0$-separating nest, if and only if $\triangleL$ is a linear
order.

\end{enumerate}
\end{theorem}

We still needed some more tools, in order to restate van Dalen and Wattel's characterization
theorem in a more elementary way. Theorem \ref{theorem-preliminary1} shows
the connection between a subbase which generates a GO-topology and two $T_0$-separating
nests with reverse orders, whose union $T_1$-separates the space.

\begin{theorem}\label{theorem-preliminary1}
Let $X$ be a set. Suppose that $\cL$ and $\cR$ are two nests on $X$. Then,
$\cL \cup \cR$ is $T_1$-separating, if and only if $\cL$ and $\cR$ are
both $T_0$-separating and $\triangleL = \triangleR$.
\end{theorem}

A key tool, for van Dalen and Wattel's solution of the Orderability Problem,
was the notion of interlocking.

\begin{definition}\label{definition - interlocking}
Let $X$ be a set and let $\mathcal{S} \subset \mathcal{P}(X)$. We say that
$\mathcal{S}$ is interlocking, if and only if
for each $T \in \mathcal{S}$, such that:
\[T = \bigcap \{S : T \subset S, S \in \mathcal{S} - \{T\}\}\] we have that:
\[T = \bigcup \{S : S \subset T, S \in \mathcal{S} - \{T\}\}.\]
\end{definition}

By Lemma \ref{Lemma - interl via min max} that follows, we clarified the relationship between
an interlocking nest and the properties of its induced order.

\begin{lemma}\label{Lemma - interl via min max}
Let $X$ be a set and let $\cL$ be a $T_0$-separating nest on $X$. Then,
the following hold for $L \in \cL$:
\begin{enumerate}

\item $L = \bigcap \{M \in \cL : L \subsetneq M\}$, if and only if $X-L$
has no $\triangleL$-minimal element.

\item $L = \bigcup \{M \in \cL : M \subsetneq L\}$, if and only if $L$
has no $\triangleL$-maximal element.

\end{enumerate}
\end{lemma}

It is immediate, from Definition \ref{definition - interlocking}, that a collection $\cL$
is interlocking, if and only if for all $L \in \cL$, either $L = \bigcup \{N \in \cL : N \subsetneq L\}$
or $L \neq \bigcap \{N \in \cL : L \subsetneq N\}$. So, we observed that Theorem \ref{theorem - for Theorem Dalen-Wattel}
and Lemma \ref{Lemma - interl via min max} therefore imply the following.

\begin{theorem}\label{theorem - before Lema for Theorem Dalen-Wattel}
Let $X$ be a set and let $\cL$ be a $T_0$-separating nest on $X$. The
following are equivalent:
\begin{enumerate}

\item $\cL$ is interlocking;

\item for each $L \in \cL$, if $L$ has a $\triangleL$-maximal element, then
$X-L$ has a $\triangleL$-minimal element;

\item for all $L \in \cL$, either $L$ has no $\triangleL$-maximal element
or $X-L$ has a $\triangleL$-minimal element.

\end{enumerate}
\end{theorem}

So, Theorem \ref{theorem - before Lema for Theorem Dalen-Wattel} is a specific version
of the notion interlocking in the case of a linearly ordered topological space, and this
gave us enough tools to prove the following alteration of van Dalen and Wattel's Theorem:

\begin{theorem}[van Dalen \& Wattel]\label{theorem - them Dalen-Wattel}
Let $(X,\mathcal{T})$ be a topological space. Then:
\begin{enumerate}

\item If $\mathcal{L}$ and $\mathcal{R}$ are two nests of open
sets, whose union is $T_1$-separating, then every $\triangleleft_{\mathcal{L}}$-order
open set is open, in $X$.

\item $X$ is a GO space, if and only if there are two nests, $\mathcal{L}$ and $\mathcal{R}$,
of open sets, whose union is $T_1$-separating and forms a subbase for $\cT$.

\item $X$ is a LOTS, if and only if there are two interlocking nests $\cL$ and $\cR$,
of open sets, whose union is $T_1$-separating and forms a subbase for $\cT$.

\end{enumerate}
\end{theorem}

\subsection{Characterizations of Ordinals.}

Ordinals, like LOTS and GO-spaces, are fundamental building blocks for
set-theoretic and topological examples. In \cite{Good-Papadopoulos} we
used properties of nests in order to characterize ordinals topologically.
To achieve this, we considered our spaces to be ``scattered by
a nest''.

\begin{definition}\label{definition-scattered-1}
A topological space $X$ is {\em scattered}, if every non-empty subset $A \subset X$
has an isolated point, i.e. for
every non-empty $A \subset X$, there exists $a \in A$ and $U$ open in $X$,
such that $U \cap A= \{a\}$.
\end{definition}

Therefore, a space $X$ is {\em scattered}, if for every non-empty $A \subset X$, there
exists $U$ open in $X$, such that $|U \cap A| = 1$.

\begin{definition}\label{definition-scattered by family}
 Let $\mathcal{S}$ be a family of subsets of a set $X$. We say
 that $X$ is {\em scattered} by $\mathcal{S}$, if and only if
 for every $A \subset X$, there exists $S \in \mathcal{S}$, such
 that $|A \cap S| = 1$.
 \end{definition}

\begin{theorem}\label{theorem - triangleL is a well-order}
Let $X$ be a set and let $\cL$ be a nest on $X$. Then, the following are equivalent:
\begin{enumerate}

\item $\cL$ scatters $X$.

\item $\triangleL$ is a well-order on $X$.

\item $\cL$ is $T_0$-separating and well-ordered by $\subset$.

\item $\cL$ is $T_0$-separating and, for every non-empty subset $A$ of $X$, there is an $a \in A$,
such that for any $x \in A$ and any $L \in \cL$, if $x \in L$, then $a \in L$.

\end{enumerate}
\end{theorem}

\begin{theorem}\label{theorem - ordinal scatterness nest}
Let $X$ be a space. The following are equivalent:
\begin{enumerate}

\item $X$ is homeomorphic to an ordinal.

\item $X$ has two interlocking nests $\cL$ and $\cR$, of open sets, whose union is a $T_1$-separating
subbase, such that $\cL$ scatters $X$.

\item $X$ has two interlocking nests $\cL$ and $\cR$, of open sets, whose union is a $T_1$-separating
subbase, one of which is well-ordered by $\subset$ or $\supset$.

\item $X$ is scattered by a nest $\cL$, of clopen sets, such that:
\begin{enumerate}

\item $L \neq \bigcup \{M : M \subsetneq L\}$, for any $L \in \cL$ and

\item $\{L-M: L,M \in \cL\}$ is a base for $X$.

\end{enumerate}

\item $X$ is scattered by a nest of compact clopen sets.

\end{enumerate}

\end{theorem}

Corollary \ref{corollary - omega1} that follows leaded us
to a characterization of the ordinal $\omega_1$, with clear
links to the well-known Pressing (or Pushing) Down Lemma
in Set Theory.

\begin{corollary}\label{corollary - omega1}
$X$ is homeomorphic to a cardinal, if and only if $X$ is scattered
by a nest $\cL$, of compact clopen sets, such that $|L|<|X|$, for each $L \in \cL$.

In particular, $X$ is homeomorphic to $\omega_1$, if and only if $X$ is uncountable
and is scattered by a nest of compact, clopen, countable sets.
\end{corollary}

\subsection{A Generalization of the Orderability Problem.}

In \cite{On-The-Orderability-Thm}, we restated Theorem \ref{theorem - them Dalen-Wattel}
via the interval topology, in the corollary that follows.

\begin{corollary}
A topological space $(X,\cT)$ is:
\begin{enumerate}

\item a LOTS, iff there exists a nest $\cL$ on $X$, such that $\cL$ is $T_0$-separating and interlocking
and also $\cT = \cT_{in}^{\triangleLq}$.

\item a GO-space, iff there exists a nest $\cL$ on $X$, such that $\cL$ is $T_0$-separating
and also $\cT = \cT_{in}^{\triangleLq}$.

\end{enumerate}
\end{corollary}

An answer to the following question will give a weaker orderability theorem.

{\bf Question:} Let $X$ be a set equipped with a transitive relation $<$ and the interval topology $\cT_{in}^\le$,
 defined via $\le$, where $\le$ is $<$ plus reflexivity. Under which necessary and sufficient conditions will $\cT_{<}$ be equal to $\cT_{in}^{\le}$?

\section{Some New Thoughts.}
\label{sec:4}

\subsection{Connectedness and Orderability.}

In this section we give a characterization of interlockingness via connectedness. This will
give a condition for a connected space to be LOTS.

\begin{definition}\label{definition-dense}
A partial order $<$, on a set X, is said to be dense if, for all $x$ and $y$ in
$X$ for which $x < y$, there exists some $z$ in $X$, such that $x < z < y$.
\end{definition}

So, given Definition \ref{definition-dense}, the next lemma follows naturally.

\begin{lemma}-
Let $X$ be a set and let $\cL$ be a nest on $X$. Then, the order $\triangleL$ is dense in $X$, if and only if for every $x,y \in X$, $x \neq y$, there exist $L,M \in \cL$,
$L \subsetneq M$, such that $x \in L$ and $y \notin M$ or $y \in L$ and $x \notin M$.
\end{lemma}

\begin{proposition}
Let $X$ be a set and let $\mathcal{L},\mathcal{R}$ be two nests
of open sets on $X$, such that $\cL \cup \cR$ creates a $T_1$-separating subbase for a topology on $X$. If $X$ is connected, with respect to the topology that is induced by the union of $\cL$ and $\cR$, then $\triangleL$ is dense in $X$.
\end{proposition}
\begin{proof}
Suppose $\triangleL$ is not dense. Then, there exist $x,y \in X$, such that $(x,y) = \emptyset$. So,
there exists $L \in \cL$, such that $x \in L$ and $y \notin L$ and there also exists $R \in \cR$,
such that $x \notin R$ and $y \in R$ and also $L \cap R = \emptyset$ and $L \cup R = X$.
So, $X$ is not connected.

\end{proof}

In Theorem \ref{theorem - before Lema for Theorem Dalen-Wattel} we described interlocking nests, in terms of maximal and minimal elements. Here we
use this result, in order to give a characterization of connected spaces via nests.

\begin{theorem}\label{theorem - LOTS via connected}
Let $X$ be a set and let $\mathcal{L},\mathcal{R}$ be two nests of open sets
on $X$, such that $\cL \cup \cR$ creates a $T_1$-separating subbase for a topology on $X$. If
$X$ is connected, with respect to the topology with subbase $\cL \cup \cR$, then $\cL$ and $\cR$ are interlocking nests.
\end{theorem}
\begin{proof}
If $\cL$ is not interlocking then, according to Theorem \ref{theorem - before Lema for Theorem Dalen-Wattel},
there exists $L \in \cL$, such that $L= (-\infty,x]$, but $X-L$ has no minimal element. The
set $L$ is open, as a subbasic element for the topology that is generated by $\cL \cup \cR$. So, for every $z \in X-L$, there exists
$z'$, such that $x \triangleL z' \triangleL z$. But, there exists $R_z \in \cR$, such that $z' \notin R_z$
and $z \in R_z$. So, $X-L = \bigcup_{z \notin L}R_z$, i.e. $R_z \cap L = \emptyset$. Thus,
$X-L$ is open and $L$ is open, hence $X$ is not connected. In a similar way, $\cR$ is interlocking, too.
\end{proof}

Theorem \ref{theorem - LOTS via connected}
permits us now to view LOTS, in the light of connectedness.

\begin{corollary}\label{corollary-connected char}
Let $X$ be a set and let $\mathcal{L},\mathcal{R}$ be two nests of open sets
on $X$, such that $\cL \cup \cR$ creates a $T_1$-separating subbase for a topology on $X$. If
$X$ is connected with respect to the topology with subbase $\cL \cup \cR$, then $X$ is a LOTS.
\end{corollary}
\begin{proof}
The proof follows immediately from the statements of Theorem \ref{theorem - them Dalen-Wattel}
and Theorem \ref{theorem - LOTS via connected}.
\end{proof}

\subsection{Powers of LOTS.}

Let $I$ be a set of indices. Let $X$ be a LOTS and let $\pi$ its $i$-th canonical projection. Here we examine properties of powers of LOTS,
linking $X$ with $X^I$ via projections.

\begin{proposition}\label{Nest-from plane to line}
Let $X$ be a LOTS and let $\mathcal{L}_{X^I}$ be a nest on $X^I$. Then,
$\pi_i(\mathcal{L}_{X^I}) =
\{\pi_i(L) : L \in \mathcal{L}_{X^I}\}$ will be a nest on $X$,
for every $i \in I$.
\end{proposition}
\begin{proof}
Let $\pi_i(L_1),\pi_i(L_2) \in \pi_i(\mathcal{L}_{X^I})$, where $L_1,L_2 \in \mathcal{L}_{X^I}$. Then, $L_1 \subset L_2$ or
$L_2 \subset L_1$, which implies that $\pi_i(L_1) \subset \pi_i(L_2)$ or $\pi_i(L_2) \subset \pi_i(L_1)$,
proving that $\pi_i(\mathcal{L}_{X^I})$ is a nest, too.
\end{proof}

\begin{proposition}\label{Nest-From line to plane}
Let $\mathcal{L}_X$ be a nest on $X$. Then,
$\pi_i^{-1}(\mathcal{L}_X) = \{\pi_i^{-1}(L) : L \in \mathcal{L}_X\}$ will be
a nest on $X^I$, for every $i \in I$.
\end{proposition}
\begin{proof}
Let $\pi_i^{-1}(L_1),\pi_i^{-1}(L_2) \in \pi_i^{-1}(\mathcal{L}_X)$. Since
$\mathcal{L}_X$ is a nest, then either $L_1 \subset L_2$ or $L_2 \subset L_1$. If $L_1 \subset L_2$, then
$\pi_i^{-1}(L_1) \subset \pi_i^{-1}(L_2)$, and if $L_2 \subset L_1$, then $\pi_i^{-1}(L_2) \subset \pi_i^{-1}(L_1)$.
Thus, $\pi_i^{-1}(\mathcal{L}_X)$ will be a nest, too.
\end{proof}

\begin{definition}\label{Weak-T0}
Let $X$ be a set, and let $\mathcal{L}_{X^I}$ be a nest on $X^I$,
satisfying the condition that if $(x_i)_{i \in I},(y_i)_{i \in I} \in X^I$, such that $x_j \neq y_j$, $j \in I$, then
there exists $L \in \mathcal{L}_{X^I}$, such that $(x_i)_{i \in I} \in L$ and $(y_i)_{i \in I} \notin L$
or $(y_i)_{i \in I} \in L$ and $(x_i)_{i \in I} \notin L$.
Then, we say that the nest $\mathcal{L}_{X^I}$ is {\em weakly $T_0$-separating}, with respect to the $j$-th variable.
\end{definition}

\begin{definition}
Let $X$ be a set and let $\mathcal{L}_{X^I}$ be a nest on $X^I$. Let also
$(x_i)_{i \in I}$, $(y_i)_{i \in I}$, be such that $x_j \neq y_j$, for a fixed $j \in I$. Then, we define $(x_i)_{i \in I} \triangleleft_{\mathcal{L}_{X^I}} (y_i)_{i \in I}$,
if there exists a set $L \in \mathcal{L}_{X^I}$, such that
$(x_i)_{i \in I} \in L$ and $(y_i)_{i \in I} \notin L$.
\end{definition}


\begin{theorem}\label{Theorem: From plane to line}
If $\mathcal{L}_{X^I}$ is a weakly $T_0-$separating nest on $X^I$, with respect
to the $j$-th variable, such that it satisfies the condition that if $(x_i)_{i \in I} \notin L \in \mathcal{L}_{X^I}$,
then $x_j \notin \pi_j(L)$, then
$\pi_j(\mathcal{L}_{X^I}) = \{\pi_j(L) : L \in \mathcal{L}_{X^I}\}$ is a
$T_0$-separating nest on $X$.

\end{theorem}

\begin{proof}
Proposition \ref{Nest-from plane to line} gives that $\pi_j(\mathcal{L}_{X^I})$ is
a nest.

For proving that $\pi_j(\mathcal{L}_{X^I})$ is $T_0$-separating, let $x_1,x_2 \in X$, such that
$x_1 \neq x_2$. Then, we form $(y_i)_{i \in I}$, $(z_i)_{i \in I}$, so that we place $x_1$ in the $j$-th position of $(y_i)_{i \in I}$ and $x_2$ in the
$j$-th position of $(z_i)_{i \in I}$. The rest $y_i$ and $z_i$ are considered arbitrary.

Since $\mathcal{L}_{X^I}$ is $T_0$-separating, with respect to the $j$-th variable,
then there exists $L \in \mathcal{L}_{X^I}$, such that $(y_i)_{i \in I} \in L$ and
$(z_i)_{i \in I} \notin L$ or $(z_i)_{i \in I} \in L$ and $(y_i)_{i \in I} \notin L$.

So, $\pi_j((y_i)_{i \in I}) = x_1 \in \pi_j(L)$ and $\pi_j((z_i)_{i \in I}) = x_2 \notin \pi_j(L)$
or $\pi_j((z_i)_{i \in I}) = x_2 \in \pi_j(L)$ and $\pi_j((y_i)_{i \in I}) = x_1 \notin \pi_j(L)$,
which proves that $\pi_j(\mathcal{L}_{X^I})$ is $T_0$-separating.
\end{proof}

\begin{remark}
Let $\mathcal{L}_{X^I}$ be a weakly $T_0$-separating nest in $X^I$. Then,
if $(y_i)_{i \in I}$, $(z_i)_{i \in I}$ have in the $j$-th position the elements
$y_j$ and $z_j$, respectively, then $(y_i)_{i \in I} \triangleleft_{\mathcal{L}_{X^I}} (z_i)_{i \in I}$
implies that $y_j \triangleleft_{\pi_j(\mathcal{L}_{X^I})} z_j$.
\end{remark}

\begin{definition}\label{Weak-T1}
Let $X$ be a set and let $\mathcal{L}_{X^I}$, $\mathcal{R}_{X^I}$ be nests
on $X^I$. Then, $\mathcal{L}_{X^I} \cup \mathcal{R}_{X^I}$ will be called {\em weakly $T_1$-separating}, with
respect to the $j$-th variable,
if and only if for every $(x_i)_{i \in I},(y_i)_{i \in I} \in X^I$, such that $x_j \neq y_j$, there exist
$L \in \mathcal{L}_{X^I}$ and $R \in \mathcal{R}_{X^I}$, such that $(x_i)_{i \in I} \in L$ and $(y_i)_{i \in I} \notin L$ and
also $(y_i)_{i \in I} \in R$ and $(x_i)_{i \in I} \notin R$.
\end{definition}

In this case, it is easy to see that $(x_i)_{i \in I} \triangleleft_{\mathcal{L}_{X^I}} (y_i)_{i \in I}$, if and only if
$(y_i)_{i \in I} \triangleleft_{\mathcal{R}_{X^I}} (x_i)_{i \in I}$.

\begin{proposition}\label{Weak-T1 from line to plane}
Let $X$ be a set and let also $\mathcal{L}_X$ and $\mathcal{R}_X$ be two nests on $X$, such that
$\mathcal{L}_X \cup \mathcal{R}_X$ is $T_1$-separating in $X$. Then, $\pi_j^{-1}(\mathcal{L}_X) \cup \pi_j^{-1}(\mathcal{R}_X)$
is weakly $T_1$-separating in $X \times X$, with respect to the $j$-th variable.
\end{proposition}
\begin{proof}
Let $(x_i)_{i \in I},(y_i)_{i \in I} \in X^I$, such that $x_j \neq y_j$. Then,
there exist $L \in \mathcal{L}_X$ and $R \in \mathcal{R}_X$, such that $x_j \in L$ and
$y_j \notin L$ and also $y_j \in R$ and $x_j \notin R$, which implies that $(x_i)_{i \in I} \in \pi_j^{-1}(L)$,
$(y_i)_{i \in I} \notin \pi_j^{-1}(R)$,
and also $(y_i)_{i \in I} \in \pi_j^{-1}(R)$ and $(x_i)_{i \in I} \notin \pi_j^{-1}(R)$. Thus, $\pi_j^{-1}(\mathcal{L}_X) \cup \pi_j^{-1}(\mathcal{R}_X)$ is weakly $T_1$-separating, with respect to the $j$-th variable.
\end{proof}

\begin{proposition}
Let $X$ be a set and let $\mathcal{L}_{X^I}$ and $\mathcal{R}_{X^I}$ be two nests in $X^I$, such that $\mathcal{L}_{X^I} \cup \mathcal{R}_{X^I}$ is weakly $T_1$-separating in $X^I$, with respect to the $j$-th variable. Let also $\mathcal{L}_{X^I}$ and $\mathcal{R}_{X^I}$
satisfy the condition that if $(x_i)_{i \in I} \notin L \in \mathcal{L}_{X^I}$, then
$\pi_j((x_i)_{i \in I}) = x_j \notin \pi_j(\mathcal{L}_{X^I})$, and if $(x_i)_{i \in I} \notin R \in \mathcal{R}_{X^I}$,
then $\pi_j((x_i)_{i \in I}) \notin \pi_j(R)$. Then, $\pi_j(\mathcal{L}_{X^I}) \cup \pi_j(\mathcal{R}_{X^I})$ is $T_1$-
separating in $X$.
\end{proposition}
\begin{proof}
Let $x_1 \neq x_2$. Then, $(y_i)_{i \in I} \neq (z_i)_{i \in I}$, where $y_j = x_1$, $z_j=x_2$,
and the rest $y_i$ and $z_i$ are arbitrary. Since $\mathcal{L}_{X^I} \cup \mathcal{R}_{X^I}$
is weakly $T_1$-separating, there exist $L \in \mathcal{L}_{X^I}$, $R \in \mathcal{R}_{X^I}$,
such that $(y_i)_{i \in I} \in L$ and $(z_i)_{i \in I} \notin L$ and also $(z_i)_{i \in I} \in R$ and
$(y_i)_{i \in I} \notin R$, which implies that $x_1 \in \pi_j(L)$ and $x_2 \notin \pi_j(R)$ and also
$x_2 \in \pi_j(R)$ and $x_1 \notin \pi_j(R)$.
\end{proof}


\begin{theorem}
Let $X$ be a set and let $\mathcal{L}_X$ be
an interlocking nest in $X$. Then, $\mathcal{M} = \{\pi_j^{-1}(L): L \in \mathcal{L}_X\}$ will be an interlocking nest in $X^I$.
\end{theorem}
\begin{proof}
Suppose $M \in \mathcal{M}$ be such that $M
  = \bigcap\{M' \in \mathcal{M}: M' \supsetneq M\}$.
By the definition of $\mathcal{M}$, there exists $L \in \mathcal{L}$ such that:
$M = \pi^{-1}_j(L) = \Pi_i \{Y_i: Y_i = X, \textrm{ if } i \neq j \textrm{ and }Y_i = L \textrm{ if } i=j\}$.
Making a similar substitution for all $M' \in \mathcal{M}$, we deduce that: $\Pi_i \{Y_i : Y_i = X, \textrm{ if } i \neq j \textrm{ and } Y_i = L, \textrm{ if }i=j\} =  \bigcap \Pi_i \{Z_i : Z_i = X, \textrm{ if } i \neq j \textrm{ and } Z_i = L', if \textrm{ i=j },\,L' \in \mathcal{L},\,L' \supset L\} = \bigcap \{W_i : W_i = X, \textrm{ if } i \neq j \textrm{ and } W_i = \bigcap\{L' \in \mathcal{L}: L' \supsetneq L \},\textrm{ if } i =j\}$. So, $L = \bigcap \{L' \in \mathcal{L} : L' \supsetneq L\}$,
which implies that $L = \bigcup \{L' \in \mathcal{L} : L' \subsetneq L\}$. Hence, $M = \Pi_i \{Y_i : Y_i = X, \textrm{ if }
i \neq j \textrm{ and } Y_i = L, \textrm{ if } i=j\} = \Pi_i \{\Theta_i : \Theta_i = X, \textrm{ if }, i \neq j
\textrm{ and } \Theta_i = \bigcup \{L' \in L : L' \subsetneq L\}, \textrm{ if } i=j\}$. So, $M = \bigcup \{M' \in \mathcal{M} : M' \subsetneq M\}$, which proves that $\mathcal{M}$ is interlocking.


\end{proof}

\begin{lemma}\label{lemma-preliminary to interlocking}
Let $X$ be a set and let $\mathcal{L}_{X^I}$ be a collection of subsets
of $X^I$. If the following condition holds: [if $(x_i)_{i \in I} \notin L$, $L \in \mathcal{L}_{X^I}$,
then $x_j \notin \pi_j(L)$],
then $\pi_j (L) \supset \bigcap \{\pi_j(L'): \pi_j(L') \supsetneq \pi_j(L)\}$ implies that $L \supset \bigcap \{L' : L' \supsetneq L\}$.
\end{lemma}
\begin{proof}
If $\bigcap\{L' : L' \supsetneq L\} \subsetneq L$, then there exists $(x_i)_{i \in I}$, such
that $(x_i)_{i \in I} \in \bigcap \{L' : L' \supsetneq L\}$ and $(x_i)_{i \in I} \notin L$. So,
$(x_i)_{i \in I} \in L'$, for every $L' \supsetneq L$, and $\pi_j ((x_i)_{i \in I})= x_j \notin \pi_j(L)$.
Thus, $x_j \in \pi_j(L')$, for all $\pi_j(L') \supsetneq \pi_j(L)$ and $x_j \notin \pi_j(L)$, which
contradicts the statement of the Lemma \ref{lemma-preliminary to interlocking}.
\end{proof}

\begin{theorem}
Let $\mathcal{L}_{X^I}$ be an interlocking nest in $X^I$. Then, $\pi_j(\mathcal{L}_{X^I})$,
$j \in I$, is an interlocking nest in $X$, if the condition in Lemma \ref{lemma-preliminary to interlocking}
holds.
\end{theorem}
\begin{proof}
We have that $\pi_j(\mathcal{L}_{X^I}) = \{\pi_j(L): L \in \mathcal{L}_{X^I}\}$. Let
\[\pi_j(L) = \bigcap \{\pi_j(L')\}: \pi_j(L') \supsetneq \pi_j(L).~~~(1)\] We will prove that
\[\pi_j(L)= \bigcup \{\pi_j(L') : \pi_j(L') \subsetneq \pi_j(L)\}~~~(2)\] or,
equivalently, we will prove that: \[\pi_j(L) \subset \bigcup \{\pi_j(L') : \pi_j(L') \subsetneq \pi_j(L)\}\]

Since $(1)$ is satisfied, we have that $\pi_j(L) \supset \bigcap \{\pi_j(L') : \pi_j(L') \supsetneq \pi_j(L)\}$ which implies,
by Lemma \ref{lemma-preliminary to interlocking}, that $L \supset \bigcap\{L' : L' \supsetneq L\}$. But since
it is always true that $L \subset \bigcap\{L' : L' \supsetneq L\}$, we have that $L = \bigcap \{L' : L' \supsetneq L\}$,
and since $\mathcal{L}_{X^I}$ is interlocking, we have that $\pi_j(L) \subset \bigcup \{\pi_j(L'): L' \subsetneq L\} \subset \bigcup \{\pi_j(L'): \pi_j(L') \subsetneq \pi_j(L)\}$, which completes the proof.
\end{proof}

\begin{theorem}
Let $X$ be a topological space and let $\mathcal{L}_{X^I}$, $\mathcal{R}_{X^I}$ be two interlocking, weakly $T_0$-separating nests in $X^I$, such that their
union, $\mathcal{L}_{X^I} \cup \mathcal{R}_{X^I}$ is weakly $T_1$-separating, with respect to the $j$-th variable. Let also for $L \in \mathcal{L}_{X^I}$ and $R \in \mathcal{R}_{X^I}$, $\mathcal{L}_{X^I}$ and $\mathcal{R}_{X^I}$ satisfy the following two conditions:
\begin{enumerate}

\item If $x_j \notin L$, then $x_j \notin \pi_j(L)$.

\item If $x_j \notin R$, then $x_j \notin \pi_j(R)$

\end{enumerate}
Then, $\pi_j(\mathcal{L}_{X^I})$ and $\pi_j(\mathcal{R}_{X^I})$
are interlocking, $T_0$-separated nests of open sets, in $X$, such that
their union, $\pi_j(\mathcal{L}_{X^I}) \cup \pi_j(\mathcal{R}_{X^I})$
is $T_1$-separating (thus, the topology of $X$ will coincide with the order topology).
\end{theorem}
\begin{proof}
We have already shown that the canonical projection of a weakly $T_0$-separating nest
is a $T_0$-separating nest, that the projection of a weakly $T_1$-separating
union of two nests of open sets is $T_1$-separating, and also that interlockingness
is preserved in a nest, if we project it via canonical projection. The only thing that remains to complete
the proof is to remark that $\pi_j$ is an open mapping, so for each $L$ open in $X^I$, $\pi_j(L)$ and $\pi_j(R)$ are open sets in $X$,
and this completes the proof.
\end{proof}

\begin{corollary}
Let $X$ be a topological space and let $\mathcal{L}$ and $\mathcal{R}$
be two $T_0$-separating, interlocking nests of open sets, in $X$, such that
$\mathcal{L} \cup \mathcal{R}$ is $T_1$-separating. Then, $\mathcal{S} =
\{\bigcap_{j_k \in J_k} \pi^{-1}_{j_k}(L \cap R)\}$, $J_k \subset I,\,L \in \mathcal{L},
R \in \mathcal{R}\}$ will be a base for a topology in $X^I$.
\end{corollary}

\subsection{LOTS and Function Spaces.}

Let $X$ and $Y$ be two sets and let $\mathcal{F}(X,Y) = \{f: f \textrm{ is a function }, f: X \to Y\}$. Then, it is known that
$\mathcal{F}(X,Y) = \Pi_{x \in X} Y_X$, where $Y_X = Y$, for all $x \in X$.

\begin{theorem}
Let $X$ and $Y$ be two sets, and let $\mathcal{F}(X,Y)$ be the function space, that consists
of all functions from $X$ to $Y$. Let also $\mathcal{L}$ be a nest on $Y$. Then, for each
$x \in X$, the set $\mathcal{L}^{x}_{\mathcal{F}(X,Y)} = \{(x,L): L \in \mathcal{L}\}$, where
$(x,L) = \{f \in \mathcal{F}(X,Y) : f(x) \in L\}$, will be a nest on $\mathcal{F}(X,Y)$.
\end{theorem}
\begin{proof}
We remark that $\mathcal{L}^{x}_{\mathcal{F}(X,Y)} = \{\pi_x^{-1}(L) : L \in \mathcal{L}\}$ is
a nest, and this proves the assertion of the theorem.
\end{proof}

\begin{remark}
Let $\mathcal{F}(X,Y)$ be a function space and let also $\mathcal{L}_Y$ and $\mathcal{R}_Y$ be
two nests on $Y$, such that $\mathcal{L}_Y \cup \mathcal{R}_Y$ is $T_1$-separating. Let also
$x \in X$ be a point in $X$. Then, $\{(x,L) : L \in \mathcal{L}_Y\} \cup \{(x,R): R \in \mathcal{R}_Y\}$
is weakly $T_1$-separating, with respect to $x$. This means that if $f,g \in \mathcal{F}(X,Y)$,
such that $f(x) \neq g(x)$, then there exist $L,R$ in $\mathcal{L}_Y,\mathcal{R}_Y$, respectively,
such that $f \in (x,L)$ and $g \notin (x,L)$ and also $g \in (x,R)$ and $f \notin (x,R)$, which is
an immediate consequence of Proposition \ref{Weak-T1 from line to plane}.

Last, but not least, the union $\bigcup_{x \in X} \mathcal{L}_{\mathcal{F}(X,Y)} \cup \bigcup_{x \in X} \mathcal{R}_{\mathcal{F}(x,y)}$ is a subbase for the point-open topology.

\end{remark}

\begin{corollary}
Let $X$ and $Y$ be two sets and let also $\mathcal{L}$ be a nest on $Y$. Then, for each $x \in X$, all the
nests of the form $\mathcal{L}^{x}=\{(x,L): L \in \mathcal{L}\}$ are interlocking.
\end{corollary}

\section{Nests and the Ring ${}^\bullet \mathbb{R}$ of Fermat Reals.}
\label{sec:5}

\subsection{A short introduction.}

The idea of the ring of Fermat Reals $\FR$ has come as a possible alternative to Synthetic
Differential Geometry (see e.g. [11,12,13,14]) and its main aim is the development
of a new foundation of smooth differential geometry for finite and infinite-dimensional
spaces. In addition, $\FR$ could play a role of a potential alternative in some certain
 problems in the field ${}^\star \mathbb{R}$ in Nonstandard Analysis (NSA), because
 the applications of NSA in differential geometry are very few.
One of the ``weak'' points of $\FR$ at the moment is
the lack of a natural topology, carrying the strong topological
properties of the line.

P. Giordano and M. Kunzinger have recently done brave steps
towards the topologization of the ring ${}^\bullet \mathbb{R}$ of
Fermat Reals. In particular, they have constructed two topologies;
the Fermat topology and the omega topology (see [11]). The Fermat topology is generated
by a complete pseudo-metric and is linked to the differentiation of non-standard
smooth functions. The omega topology is generated by a complete metric and
is linked to the differentiation of smooth functions on infinitesimals.
Although both topologies are very useful in developing infinitesimal
instruments for smooth differential geometry, none of these
two topologies aims to characterize the Fermat real line from an
order-theoretic perspective. In fact, neither makes the space $T_1$, while
an appropriate order-topology would equip the Fermat Real Line
with the structure of a monotonically normal space, at least. The
possibility to define a linear order relation on $\FR$, so that it
can be viewed as a LOTS (linearly ordered topological space) can
be considered important, because $\FR$ is an alternative mathematical
model of the real line, having some features with respect to
applications in smooth differential geometry and mathematical physics.
It is therefore natural to ask whether for $\FR$ peculiar characteristics
of $\mathbb{R}$ hold or not.

In this section we will focus in the order relation which is stated
in [12], and we will interpreted through nests.

As we shall see in Definition \ref{definition-equivalens relation Fermat Real},
the idea of the formation of $\FR$ starts with an equivalence relation in
the little-oh polynomials, where $\FR$ is the quotient space under this
relation. This treatment permits us to view these little-oh polynomials as
numbers.

\subsection{Definitions.}

\begin{definition}
A little-oh polynomial $x_t$ (or $x(t)$) is an ordinary set-theoretical function, defined as follows:
\begin{enumerate}

\item $x : \mathbb{R}_{\ge 0} \to \mathbb{R}$ and

\item $x_t = r + \sum_{i=1}^k \alpha_i t^{a_i} + o(t)$, as $t \to 0^+$, for
suitable $k \in \mathbb{N}$, $r,\alpha_1,\cdots,\alpha_k \in \mathbb{R}$ and $a_1,\cdots,a_k \in \mathbb{R}_{\ge 0}$.
\end{enumerate}
\end{definition}
The set of all little-oh polynomials is denoted by the symbol $\mathbb{R}_o[t]$. So, $x \in \mathbb{R}_o(t)$,
if and only if $x$ is a polynomial function with real coefficients, of a real variable $t \ge 0$, with
generic positive powers of $t$ and up to a little-oh function $o(t)$, as $t \to 0^+$.

\begin{definition}\label{definition-equivalens relation Fermat Real}
Let $x,y \in \mathbb{R}_o[t]$. We declare $x\sim y$ (and we say $x=y$ in ${}^\bullet \mathbb{R}$), if and only
if $x(t) = y(t) + o(t)$, as $t \to 0^+$.
\end{definition}

The relation $\sim$ in Definition \ref{definition-equivalens relation Fermat Real} is an equivalence
relation and ${}^\bullet \mathbb{R} :=\mathbb{R}_o[t]/\sim$.

A first attempt to define an order in ${}^\bullet \mathbb{R}$ has come from Giordano.

\begin{definition}[Giordano]\label{definition-order Giordano}
Let $x,y \in {}^\bullet \mathbb{R}$. We declare $x \le y$, if and only if there exists $z \in {}^\bullet \mathbb{R}$,
such that $z = 0$ in ${}^\bullet \mathbb{R}$ (i.e. $\lim_{t \to 0^+} z_t/t = 0$) and for
every $t \ge 0$ sufficiently small, $x_t \le y_t + z_t$.
\end{definition}

For simplicity, one does not use equivalence relation but works with an equivalent language of
representatives. If one chooses to use the notations of [12], one has to note that
Definition \ref{definition-order Giordano} does not depend on representatives.

As the author describes in [12], the order relation in NSA admits all formal properties
among all the theories of (actual) infinitesimals, but there is no good dialectic of these
properties with their informal interpretation. In particular, the order in ${}^\star \mathbb{R}$
inherits by transfer all the first order properties but, on the other hand, in the quotient
field ${}^\star \mathbb{R}$ it is difficult to interpret these properties of the order
relation as intuitive properties of the corresponding representatives. For example, a geometrical
interpretation like that of $\FR$ seems not possible for ${}^ \star \mathbb{R}$. Definition \ref{definition-order Giordano}
provides a clear geometrical representation of the ring $\FR$ (see, for instance, section 4.4 of [12]).

\subsection{The Fermat Topology and the omega-topology on $\FR$.}

A subset $A \subset \FR^n$ is open in the Fermat topology, if it can be written as
$A = \bigcup \{{}^\bullet U \subset A: U \textrm{ is open in the natural topology in }\mathbb{R}^n\}$.
Giordano and Kunzinger describe this topology as the best possible one for sets having a ``sufficient
amount of standard points'', for example ${}^\bullet U$. They add that this connection between the
Fermat topology and standard reals can be glimpsed by saying that the monad $\mu(r) := \{x \in \FR : \textrm{ standard part of }x = r\}$ of a real $r \in \mathbb{R}$
is the set of all points which are limits of sequences with respect to the Fermat topology. However it is obvious that in sets of infinitesimals there is a need for constructing a (pseudo-)metric generating a finer topology that the authors call the omega-topology (see \cite{Giordano-Kunziger}).
Since neither the Fermat nor the omega-topology are Hausdorff when restricted to $\FR$ and since each
of them describes sets having a ``sufficient amount'' of standard points or infinitesimals, respectively,
there is a need for defining a natural topology on $\FR$ describing sufficiently all Fermat reals
and carrying the best possible properties.

\subsection{Interlocking Nests on ${}^\bullet \mathbb{R}$.}


\begin{theorem}
The pair $({}^\bullet \mathbb{R},<_F)$, where $<_F$ is defined as follows:

\[
x<_Fy \Leftrightarrow \begin{cases}
\exists\,\{k \in \FR : k \le l\}, \textrm{ some } l \in \FR, \textrm{ such that } x \in \{k \in \FR : k \le l\} \not\ni y, l \in \FR
 \\
or\\
x = \max\{k \in \FR : k \le l\}, \textrm{ some } l \in \FR \textrm{ and } \exists h \in \FR : h>0,\,y = x + h\\
or\\
y = \min \{k \in \FR : l \le k\}, \textrm{ some } l \in \FR \textrm{ and } \exists h \in \FR : h>0,\,x = y - h
\end{cases}
\]
where $x,y$ are distinct Fermat reals, is a linearly ordered set.
\end{theorem}

\begin{proof}
The order of Definition \ref{definition-order Giordano} gives two nests, namely
the nest $\cL$, which consists of all sets $L = \{k \in \FR: k \le l\}$, some $l \in \FR$ and the nest $\cR$,
which consists of all sets $R = \{k \in \FR : l \le k\}$, some $l \in \FR$.
In addition, we have that
$\triangleLq = \trianglerighteq_{\cR} = \le$.

We remark that, for any $L \in \cL$ (respectively for any $R \in \cR$), $L$ (resp. $R$) has a $\triangleLq$-maximal element (resp. $\triangleRq$-maximal element for $R$), such that $X-L$ has no $\triangleLq$-minimal element (resp. $X-R$ has no $\triangleRq$-minimal element). So,
 neither $\cL$ nor $\cR$ are interlocking.

Now, for all $L=\{k \in \FR : k \le l\} \in \cL$, some $l \in \FR$, let $x_L$ denote the $\triangleLq$-maximal element of $L$ and for all $R=\{k \in \FR : l \le k\} \in \cR$, some $l \in \FR$
let $y_R$ denote the $\triangleLq$-minimal element of $R$.

Furthermore, for each $L \in \cL$ choose $x_L^+ \in \FR$  and
for each $R \in \cR$ choose $y_R^- \in \FR$, where $x_L^+$ and $y_R^-$ are distinct points in ${}^\bullet \mathbb{R}$,
and define a map $p : {}^\bullet \mathbb{R} \to {}^\bullet \mathbb{R}-(\{x_L^+ : L \in \cL\} \cup \{y_R^- : R \in \cR\})$, as follows:
\[
p(x)=\begin{cases}
x,&\text{if }x\in {}^\bullet \mathbb{R}-(\{x_L^+ : L \in \cL\} \cup \{y_R^- : R \in \cR\})\\
x_L,&\text{if }x=x_L^+\\
y_R,&\text{if }x=y_R^-
\end{cases}
\]

Now, define an order $<_F$ on ${}^\bullet \mathbb{R}$, so that:
\[
x<_Fy \Leftrightarrow \begin{cases}
p(x) \triangleL  p(y) \\
or\\
x=x_L \textrm{ and } y = x_L^+\\
or\\
x=y_R^- \textrm{ and } y = y_R
\end{cases}
\]

Obviously, $<_F$ is a linear order and the restriction
of $<_F$ to ${}^\bullet \mathbb{R}-(\{x_L^+ : L \in \cL\} \cup \{y_R^- : R \in \cR\})$ equals $\triangleLq$, the order in Definition \ref{definition-order Giordano}. In addition, we can set $x_L^+ = x_L + h$, where $h$ is not zero in $\FR$ and $h>0$, that is,
 $\lim_{t \to 0^+} h_t/t \neq 0$ and, respectively, we set $x_R^- = x_R-h$, and this completes the proof.
\end{proof}

\begin{theorem}
$\FR$ equipped with the order topology from $<_F$ is a LOTS.
\end{theorem}
\begin{proof}
We will now show that the topology $\cT$ on ${}^\bullet \mathbb{R}-(\{x_L^+ : L \in \cL\} \cup \{y_R^- : R \in \cR\})$ coincides with the
subspace topology on ${}^\bullet \mathbb{R}-(\{x_L^+ : L \in \cL\} \cup \{y_R^- : R \in \cR\})$ that is inherited from the $<_F$-order topology
on ${}^\bullet \mathbb{R}$.

But, since $\cL \cup \cR$ forms a subbasis for $\cT$, that consists of two nests, every
set in $\cT$ can be written as a union of sets of the form $L \cap R$, where
$L \in \cL$ and $R \in \cR$. It suffices therefore to show that every $L \in \cL$ and
$R \in \cR$ can be written as the intersection of an order-open set with ${}^\bullet \mathbb{R}-(\{x_L^+ : L \in \cL\} \cup \{y_R^- : R \in \cR\})$.
But this is always true, since if $L \in \cL$, with $\triangleLq$-maximal element $x_L$, then
$L = {}^\bullet \mathbb{R}-(\{x_L^+ : L \in \cL\} \cup \{y_R^- : R \in \cR\}) \cap \{ x \in \FR: x<_F x_L^+\}$.

The argument for $R \in \cR$ is similar, and this completes the proof.
\end{proof}

\subsection{Remarks.}

\begin{enumerate}
\item The order topology $\cT_{<_F}$ equals the topology $\cT_{\cL_{<_F} \cup \cR_{<_F}}$,
where $\cL_{<_F}=\{k \in \FR : k <_F l\}$, some $l \in \FR$ and $\cR_{<_F}=\{k \in \FR : l <_F k\}$, some $l \in \FR$.
This is because $\cL_{<_F} \cup \cR_{<_F}$ $T_1$-separates $\FR$ and both $\cL_{<_F}$ and $\cR_{<_F}$ are interlocking nests. So, unlike
the GO-space topology $\cT_\le$ on $\FR$, where $\cT_\le \subset \cT_{\cL \cup \cR}$, $<_F$ provides a natural extension
of the natural linear order of the set of real numbers to the Fermat real line and the order topology from $<_F$ can
be completely described via the nests $\cL_{<_F}$ and $\cR_{<_F}$.

\item Viewing the Fermat real line as a LOTS and working with nests $\cL_{<_F}$ and $\cR_{<_F}$, one can now
define the product topology for $\FR^n$, some positive integer $n$, or even more generaly for $\Pi_{i \in I} \FR_i$, some arbitrary indexing
set $I$, in the usual way via the subbasis $\pi_{j_0}^{-1}(A_{j_0}) = \Pi_{i \in I}\{\FR_i : i \neq j_0\} \times A_{j_0}$,
where $A_{j_0}$ is an open subset in the coordinate space $\FR_{j_0}$ in the order topology $\cT_{<_F}$ and
$\pi_i : \Pi_{i \in I}\FR_i \to \FR_i$ the projection.

\item In this way one can define continuity for
any function $f$ from a topological space $Y$ into the product space $\Pi_{i \in I} \FR_i$ via
the continuity of $\pi_i \circ f :Y \to \FR_i$.

\item The neight of $\FR$ is $2$ and the neight of $\FR^n = n+1$ (see \cite{Will-Brian}).
Using the product topology, as stated in Remark (2), we use four nests in order to
define -for example- the topology in $\FR^2$, but since the neight of $\FR^2$ is $3$,
one can define a topology using three nests exclusively.

\end{enumerate}

\subsection{Questions.}

\begin{enumerate}
\item As a LOTS, $(\FR,<_F)$ has rich topological properties. It is, for example,
a monotone normal space. It would be interesting though to have an extensive study on
the metrizability of this space. It is known that in a GO-space the terms
metrizable, developable, semistratifiable, etc. are equivalent (see \cite{Faber-metrizability} and \cite{Encyclopedia}).
The real line (i.e. the set of all standard reals, from the point of view of $\FR$)
is a developable LOTS and this is equivalent to say that it is also a metrizable LOTS.
Is $(\FR,\cT_{<_F})$ developable?

\item Which of the subspaces of $(\FR,\cT_{<_F})$ are developable?
\end{enumerate}

Since any sequence $x_1,x_2,\cdots$ of points in $\Pi_{i \in I} \FR_i$ will converge to a point
$x \in \Pi_{i \in I}\FR_i$, iff for every projection $\pi_i: \Pi_{i \in I} \FR_i \to \FR_i$ the
sequence $\pi_i(x_1), \pi_i(x_2),\cdots$ converges to $\pi_i(x)$ in the coordinate space $\FR_i$,
any answer to the above questions will be foundamental towards our understanding of convergence in the ring
of Fermat Reals.



\begin{thebibliography}{99.}%

\bibitem{Dalen-Wattel} J. van Dalen and E. Wattel. A Topological Characterization of Ordered Spaces.
{\it General Topology and Appl.}, 3:347-354, 1973.

\bibitem{Purisch-History} S. Purisch. A History of Results on Orderability and Suborderability.
{\it Handbook of the History of General Topology}, Vol. 2 (San Antonio, TX, 1993), volume 2 of Hist.
Topol., pages 689-702. Kluwer Acad. Publ., Dordrecht, 1998.

\bibitem{Good-Papadopoulos} C. Good and K. Papadopoulos. A Topological Characterization of Ordinals:
van Dalen and Wattel revisited. {\it Topology Appl.}, 159:1565-1572, 2012.

\bibitem{compendium} Gerhard Gierz, Karl Heinrich Hofmann, Klaus Keimel, Jimmie D. Lawson, Michael W. Mislove
and Dana S. Scott. A Compendium of Continuous Lattices, {\it Springer-Verlag}, Berlin 1980.

\bibitem{good-bijective-preimages} Chris Good. Bijective Preimagies of $\omega_1$. {\it Topology Appl.}, 75(2):125-142, 1997.

\bibitem{Lutzer-LOTS} David J. Lutzer. Ordered Topological Spaces. {\it Surveys in General Topology}, 247-295, Academic
Press, New York, 1980.

\bibitem{de-groot} J. de Groot and P.S. Schnare. A topological characterization of products of compact totally ordered spaces.
{\it General Topology and Appl.}, 2: 67-73, 1972.
  
\bibitem{Reed} G.M. Reed. The intersection topology w.r.t. the real line and the countable ordinals.
{\it Trans. Amer. Math. Soc.}, Vol. 247, number 2, pages 509-520, 1986.

\bibitem{van-douwn-misnamed} Eric K. van Douwen. On Mike's misnamed intersection topologies.
{\it Topology Appl.}, Vol. 51, number 2, 197-201. 1993.

\bibitem{On-The-Orderability-Thm} Kyriakos Papadopoulos, On the Orderability Problem and the Interval Topology, in {\it Topics in Mathematical Analysis and Applications}, Optimization and Its Applications Springer Series, T. Rassias and L. Toth Eds, Springer Verlag, 2014.

\bibitem{Giordano-Kunziger} Paolo Giordano and Michael Kunzinger. Topological and Algebraic
Structures on the Ring of Fermat Reals. {\it Israel Journal of Mathematics}, 193 (2013), 459-505.

\bibitem{Fermat-Reals-Prebook} Paolo Giordano. Fermat Reals: Nilpotent Infinitesimals and Infinite
Dimensional Spaces. {\it Prebook} (http://arxiv.org/abs/0907.1872).

\bibitem{Giordano-6} Paolo Giordano. Fermat reals: infinitesimals without Logic. {\it Miskolc Mathematical Notes}, Vol. 14 (2013), No. 3, pp. 65–80.

\bibitem{Giordano-7} Paolo Giordano. The ring of fermat reals. {\it Advances in Mathematics} 225 (2010), pp. 2050-2075.

\bibitem{Encyclopedia} K.P. Hart, J.-I. Nagata, J.E. Vaughan. Encyclopedia of General Topology. {\it Elsevier Science and Technology Books} (2014).

\bibitem{Faber-metrizability} M.J. Faber. Metrizability in Generalized Ordered Spaces. Matematisch Centrum, Amsterdam 1974.

\bibitem{Will-Brian} W.R. Brian. Neight: The Nested Weight of a Topological Space. {\it to appear in Topology Proceedings}.

\bibitem{Synthetic-Dif-Geo} A. Kock. Synthetic Differential Geometry. volume 51 of {\it London Math, Soc. Lect. Note Series}, Cambridge Univ. Press, 1981.

\bibitem{Basic-Synth-Dif-Geo} R. Lavendhomme. Basic Concepts of Synthetic Differential Geometry. {\it Kluwer Academic Publishers}, Dordrecht, 1996.

\bibitem{Infinitesimal-Analysis} I. Moerdijk, G.E. Reyes. Models for Smooth Infinitesimal Analysis. {\it Springer}, Berlin, 1991.

\bibitem{My-Paper-in-QA} Kyriakos Papadopoulos. On Properties of Nests: Some Answers and Questions. {\it Questions
and Answers in General Topology} (to appear)


\end{thebibliography}
\end{document}